\documentclass{amsart}
\usepackage{amsmath, amsfonts}
\usepackage{amssymb}
\usepackage{mathtools}
\usepackage{tikz-cd}
\usepackage[english]{babel}

\usepackage[vcentermath]{youngtab}

\usepackage{amsthm}
\usepackage{upgreek}
\usepackage{hyperref}
\usepackage{float, textcomp}
\usepackage{hyperref}
\usepackage{chngpage}
\usepackage{graphicx}
\usepackage{float}
\usepackage{caption}
\usepackage{cleveref}
\usepackage{array}

\theoremstyle{plain}
\newtheorem{theorem}{Theorem}

\newtheorem{prop}[theorem]{Proposition}
\newtheorem{lemma}[theorem]{Lemma}
\newtheorem{corollary}[theorem]{Corollary}

\theoremstyle{definition}

\newtheorem{conjecture}[theorem]{Conjecture}

\numberwithin{theorem}{section} 

\theoremstyle{remark}
\newtheorem*{remark}{Remark}

\DeclareMathOperator{\Tr}{Tr}

\DeclareMathOperator{\Gal}{Gal}

\newcommand{\Fp}{\mathbb{F}_p}

\newcommand{\Fq}{\mathbb{F}_q}
\newcommand{\Fqtimes}{\mathbb{F}_q^\times}

\newcommand{\Cp}{\mathbb{C}_p}

\newcommand{\Qp}{\mathbb{Q}_p}
\newcommand{\Qq}{\mathbb{Q}_q}

\newcommand{\Zp}{\mathbb{Z}_p}
\newcommand{\Zq}{\mathbb{Z}_q}

\newcommand{\Zptimes}{\mathbb{Z}_p^\times}
\newcommand{\Zqtimes}{\mathbb{Z}_q^\times}

\newcommand{\teich}{Teichm\"{u}ller\ }

\DeclareMathOperator{\NP}{NP}

\DeclareMathOperator{\ordp}{ord}

\newcommand{\ord}[1]{\ordp_p{#1}}

\newcommand{\floor}[1]{\left \lfloor #1 \right \rfloor}
\newcommand{\ceil}[1]{\left \lceil #1 \right \rceil}

\date{\today}
\title{On a Fiber Conjecture of Wan}
\author{Matthew Schmidt}
\email{mwschmid@buffalo.edu}

\begin{document}

\begin{abstract}
For a prime $p$ and $p$-power $q$, let $f(x)\in\mathbb{F}_q[x]$ with $\deg f$ coprime to $p$. As $\lambda$ varies in $\overline{\mathbb{F}_p^\times}$, Wan has conjectured that the $p$-adic Newton polygon of the corresponding Artin-Schreier curve given by $\lambda f$ is constant. That is, 
\[
	\textrm{NP}(f) = \textrm{NP}(\lambda f).
\]
In this paper, we prove this conjecture when $\lambda\in\mathbb{F}_p^\times$ and provide a detailed counterexample showing it is false in general. 
\end{abstract}

\maketitle 


\section{Introduction}

Let $p$ be a prime and $q=p^a$ a $p$-power, $a\geq 1$. Fix a primitive $p$-th root of unity $\zeta_p$ and let \[
f(x) = \sum_{i=0}^d a_ix^i\in\Fq[x],
\] 
with $\gcd(d,p)=1$.  For $k\geq 1$, attach to $f(x)$ the exponential sum:
\[
S_f^*(k) = \sum_{x\in\widehat{\mathbb{F}_{q^k}^\times}}\zeta_p^{\Tr_{\mathbb{Q}_{q^k}/\Qp}(\widehat{f}(x))}\in \Qp[\zeta_p], 
\]
where $\hat{\ }$ denotes the the Teichm\"uller lift of $\mathbb{F}_{q^k}^\times$ to $\mathbb{Z}_{q^k}^\times$ and $\hat{f}(x)$ is the coefficient-wise lift of $f(x)$ to $\Zq[x]$. To $S_f^*(k)$ we can associate an $L$-function:
\[
L_f^*(s) = \exp(\sum_{k=1}^\infty S_f^*(k)\frac{s^k}{k})\in 1+s\Qp[\zeta_p][[s]],
\]
and it is well-known that $L_f^*(s)$ is a polynomial of degree $d$. If $L_f^*(s) =\sum_{i=0}^{d}c_is^i$, the Newton polygon of $L_f^*(s)$, $\NP(f)$, is defined as the lower convex hull of the points $\{(n, \ord c_n)\}_{n=0}^{d}$, and it will be the main focus of this paper.

 In \cite{wanexpo} (Conjecture 8.15), Wan describes the following open problem:
\begin{conjecture}
Let $f(x)\in\Fq[x_1,x_2,\cdots,x_n]$ be non-degenerate. Then, the Newton polygon $\NP(\lambda f(x))$ is independent of the non-zero parameter $\lambda\in\overline{\Fp^\times}$.
\end{conjecture}
See also \cite{wantalk}. In the one-dimensional case, we first show that the conjecture is true when $\lambda$ is constrained to $\Fp^\times$:
\begin{theorem}\label{thm:a1}
For $f(x)\in\Fq[x]$ and $\lambda\in\Fp^\times$, the Newton polygon is independent of $\lambda$:
\[
	\NP(f)=\NP(\lambda f).
\]
\end{theorem}
We also show that when $a>1$ and $\lambda\in\Fqtimes$, the conjecture is false:
\begin{theorem}\label{thm:counterexample}
Let $p=5$, $a=2$ and $f(x)=x^8+x^6+x^2$. Fix $\xi$ to be a  generator of $\Fq/\Fp$ so that $\Fq = \Fp(\xi)$. Then, when $\lambda = \xi+2$, 
\[
	\NP(f)\neq\NP(\lambda f).
\]
\end{theorem}
To find this particular counterexample, we used Julia \cite{Julia} with the Nemo \cite{Nemo} computer algebra package to implement Lauder and Wan's algorithm \cite{LauderWan}  computing the $p$-adic Newton polygon of the $L$-function.  We searched starting with low degree $f$, working up until we found counterexamples, checking polynomials $f(x)$ with $\Fp$-coefficients only. The polynomial $f(x)$ from Theorem~\ref{thm:counterexample} was not the only counterexample we discovered, but it was of the lowest degree. For example, when $p=5$ and $a=2$, both $x^{11}+x^7$ and $x^9+x^7+x$ give Newton polygons that change for certain values of $\lambda$. We do not know if there are lower degree counterexamples for $p=5$ since our search was not exhaustive. See \url{https://github.com/exponentialsums/fiber} for all code used in this paper.

More generally, we show the following result limiting the values of $\lambda$ under which the Newton polygon can change:
\begin{theorem}\label{thm:main3}
For $f(x)\in\Fq[x]$, if $\lambda,\mu\in\overline{\Fp^\times}$ are such that $\lambda^{p-1}=\mu^{p-1}$, then $\NP(\lambda f)=\NP(\mu f)$.
\end{theorem}

We begin by proving the conjecture when $\lambda\in\Fp^\times$ and then compute the counterexample with the aid of Julia/Nemo. While we did our initial search using Lauder and Wan's algorithm, we do not use it in the counterexample below. Instead, we use Julia to compute three traces and determine the Newton polygons otherwise by hand in order to illuminate what causes the Newton polygon to change. While we have not found a classification of those $f(x)$ for which the conjecture holds, a key obstruction seems to be the nontrivial Galois action and the consequent collapse that occurs in Proposition~\ref{prop:even_trace}. See Section~\ref{section:xi2} for details.

\section{$(\lambda,\mu)$-invariant Polynomials}

Let $E(x)= \exp(\sum_{i=0}^\infty\frac{x^{p^i}}{p^i}) = \sum_{i=0}^\infty u_i x^i\in\Zp [[x]]$ be the Artin-Hasse exponential function and take $\pi\in\Cp$ to be such that $E(\pi)=\zeta_p$ with $\ord\pi = \frac{1}{p-1}$.

For $h(\pi)\in\Qq[[\pi]]$ and $\lambda, \mu\in\overline{\Fp^\times}$, we say $h(\pi)$ is $(\lambda,\mu)$-invariant if $\ord(h(\hat{\lambda}\pi))=\ord(h(\hat{\mu}\pi))$.  For any $h(\pi)$, we can uniquely put $h(\pi)$ in the form:
\[
	h(\pi) = \sum_{i=0}^{p-2}\pi^i H_i(\pi),
\]
where $H_i\in\Qq[[\pi^{p-1}]]$.  Further, in this form,  it is clear that the $p$-adic valuation depends entirely on the $H_i$ terms:
\begin{lemma}\label{lemma:hord}
For any $h(\pi)\in\Qq[[\pi]]$,  if we write $h(\pi)=\sum_{i=0}^{p-2} \pi^i H_i(\pi)$ with $H_i\in\Qq[\pi^{p-1}]$, then:
\[
	\ord h(\pi) = \min_{i = 0, \cdots, p-2} \left ( \frac{i}{p-1} + \ordp_pH_i\right).
\]
\end{lemma}
\begin{proof}
Observe that since $\ord\pi = \frac{1}{p-1}$ and $H_i\in\Qq[\pi^{p-1}]$, we have for each $i,j\in\{0,\cdots, p-2\}$ that $\ord(\pi^iH_i) \in \frac{i}{p-1} + \mathbb{Z}$. Hence for distinct $i,j\in \{0,\cdots, p-2\}$, we must have $\ord(\pi^iH_i)\neq \ord(\pi^jH_j)$. The lemma follows.
\end{proof}

\begin{prop}\label{prop:laminv}
Let $\lambda,\mu\in\overline{\Fp^\times}$ and $h(\pi)\in\Qq[[\pi]]$. If $\mu^{p-1}=\lambda^{p-1}$, then $h$ is $(\lambda,\mu)$-invariant.
\end{prop}
\begin{proof}
Writing $h(\pi)$ as in Lemma~\ref{lemma:hord}, we see that:
\[
	h(\hat{\lambda}\pi)=\sum_{i=0}^{p-2} \hat{\lambda}^i\pi^i H_i(\hat{\lambda}\pi).
\]
By assumption, $\hat{\lambda}^{p-1}=\hat{\mu}^{p-1}$,  and so $H_i(\hat{\lambda}\pi) = H_i(\hat{\mu}\pi)$ because $H_i\in\Qp[[\pi^{p-1}]]$. Thus, 
 by Lemma~\ref{lemma:hord} along with the fact that $\ord(\hat{\lambda}) = 0=\ord(\hat{\mu})$:
\begin{align*}
	\ord h(\hat{\lambda}\pi) &=  \min_{i = 0, \cdots, p-2} \left ( \frac{i}{p-1} + \ordp_pH_i(\hat{\lambda}\pi)\right)\\
	&= \min_{i = 0, \cdots, p-2} \left ( \frac{i}{p-1} + \ordp_pH_i(\hat{\mu}\pi)\right) = \ord h(\hat{\mu}\pi).
\end{align*}
\end{proof}

This proposition has the additional consequence that any polynomial in $\Qq[[\pi]]$ is $(\lambda,1)$-invariant when $\lambda\in\Fp^\times$:
\begin{corollary}\label{corr:mu_1}
For any $\lambda\in\Fp^\times$ and $h(\pi)\in\Qq[[\pi]]$,  $\ord(h(\hat{\lambda}\pi))=\ord(h(\pi))$.
\end{corollary}

\section{The Case $\lambda\in {\Fp^\times}$}

In this section we prove Theorem~\ref{thm:a1} and Theorem~\ref{thm:main3}.  Let $f(x) = a_dx^d + a_{d-1}x^{d-1}+\cdots + a_0\in\Fq[x]$ with $\gcd(d,p)=1$ and $a_d\neq 0$, and define
\begin{align*}
 F_f &= \prod_{i=0}^{d} E(\pi \hat{a}_i x^i) 
   = \prod_{i=0}^{d} \left (\sum_{k=0}u_k \pi^k \hat{a}_i^k x^{ik}\right ) = \sum_{i=0}^\infty F_{f,i}(\pi) x^i,
\end{align*}
for some $F_{f,i}(\pi)\in \Zp[\pi]$.  For $i\geq 0$ and $\vec{a}=(a_0,\cdots, a_d)$, define an indexing set:
\[
 \mathbb{I}_i(\vec{a}) = \{(n_0,\cdots, n_d)\in\mathbb{Z}_{\geq 0}^{d+1} :
  \sum_{k=0}^dn_k k = i \textrm{ such that } n_k\neq 0 \iff a_k\neq 0\}.
\]
For $\vec{n} = (n_0,\cdots, n_d)\in\mathbb{I}_i(\vec{a}) $, write $|\vec{n}| = \sum_{k=0}^d n_k$. We will also write $u_{\vec{n}} = \prod_{k=0}^d u_{n_k}$ and $\vec{a}^{\vec{n}} = \prod_{k=0}^d \hat{a}_k^{n_k}$.
A standard computation yields:
\begin{lemma} \label{lemma:Fi}
For $i\geq 0$,
\[
F_{f,i}(\pi) = \sum_{\vec{n}\in\mathbb{I}_i(\vec{a})} u_{\vec{n}}\vec{a}^{\vec{n}}\pi^{|\vec{n}|}.
\]
\end{lemma}

Define a Banach space 
\[
	B = \left \{\sum_{i=0}^\infty b_i \pi^{i/d}x^i | b_i\in\Zq[[\pi^{i/d}]]\right \},
\]
with formal basis $\{\pi^{i/d}x^i\}_{i\geq 0}$, 
and let $\tau$ denote the Frobenius, $\Gal(\Qq/\Qp) = \langle\tau\rangle$, with $\tau$ acting on $B$ coefficient-wise fixing $x$ and $\pi$.  
Define the mapping:
\begin{align*}
	U_p:B &\to B\\
	 \sum_{k=0}^\infty b_ix^i&\mapsto \sum_{k=0}^\infty b_{pk}x^k,
\end{align*}

and let $\phi_f$ be the operator on $B$ given by 
\[	
	\phi_f = (\tau^{-1}\circ U_p \circ F_f)^a,
\]
where $F_f$ acts on $B$ by multiplication.
It is a standard result of Dwork Theory (see \cite{Koblitz}, $\S V.3$, for more details) that $\phi_f$ is completely continuous and if we define the characteristic function\footnote{For the $C$ and $L$-functions, we adopt notation from \cite{DWX}, so that the * denotes the function is defined over the torus rather than the affine line.}:
\[
	C_f^*(\pi, s)  = \det(1-\phi_f s | B/\Zq[[\pi^{1/d}]]), 
\]
then the $L$-function may be given by:
\begin{align}\label{line:Lfunc_C}
	L_f^*(\pi,s) = \frac{C_f^*(\pi,s)}{C_f^*(\pi,qs)}. 
\end{align}
Rewriting this relationship yields:
\begin{align}\label{line:C_func_key}
	C_f^*(\pi,s) &= \prod_{i=0}^\infty L_f^*(\pi,q^is), 
\end{align}
and so the first $d$-slopes of the $p$-adic Newton polygon of $C_f^*(\pi,s)$ are exactly $\NP(f)$. Therefore, to compute $\NP(f)$, it suffices to compute the first $d$ slopes of $\det(1-\phi_f s) = \sum_{i=0}^{\infty}C_{f,i}(\pi)s^i$.

For an operator on $B$, $\psi$, and a basis of $B$, $S$, let $M_S(\psi)$ be the matrix of $\psi$ with respect to the basis $S$. 

An easy computation shows that $M_{R}(U_p\circ F_f) = (F_{f, pi-j}(\pi)\cdot\pi^{(j-i)/d})_{i,j\geq 0}$, and
so if $A_f = M_{R}(U_p\circ F_f)$, $\phi_f$ has matrix 
\[
	M_{R}(\phi_f) = A_f\tau^{-1}A_f\cdots \tau^{-(a-1)}A_f.
\]

\begin{lemma}\label{lemma:cpi}
For $i\geq 0$, we have ${C_{\lambda f, i}(\pi)} = {C_{f,i}(\hat{\lambda}\pi)}$.
\end{lemma}
\begin{proof}
For any $j\geq 0$, it is easy to see that since $\lambda\neq 0$, $\mathbb{I}_j(\vec{a})=\mathbb{I}_j(\lambda\vec{a})$, and hence by Lemma~\ref{lemma:Fi}, 
\[
	F_{\lambda f, j}(\pi) =  \sum_{\vec{n}\in\mathbb{I}_j(\vec{a})} u_{\vec{n}} \prod_{k=0}^d (\hat{\lambda} \hat{a}_k)^{n_k}\pi^{|\vec{n}|} =  \sum_{\vec{n}\in\mathbb{I}_j(\vec{a})} u_{\vec{n}} \vec{a}^{\vec{n}}(\hat{\lambda} \pi)^{|\vec{n}|}= F_{f, j}(\hat{\lambda}\pi).
\]

Let $R_1$ be the basis of $B$ given by $\{x^i\}_{i\geq 0}$. By Boyarsky \cite{boyarsky}, Step 3 on p.136, we are able to compute the trace of $\phi_f$ over the fixed basis $R_1$, rather than $R$, even though $\phi_f$ is not completely continuous over $R_1$. Observe that over $R_1$, the operator $U_p\circ F_f$ has matrix $A_{f,1}(\pi) = (F_{f,pi-j}(\pi))_{i,j\geq 0}$,
and so 
\[
	A_{\lambda f, 1}(\pi) = (F_{\lambda f,pi-j}(\pi))_{i,j\geq 0} = (F_{f,pi-j}(\hat{\lambda}\pi))_{i,j\geq 0} = A_{f,1}(\hat{\lambda}\pi).
\]
Hence if we let $M_{f,1}(\pi) = M_{R_1}(\phi_f)$, 
\[
	M_{\lambda f, 1}(\pi) = \prod_{i=0}^{a-1}\tau^{-i}A_{\lambda f,1}(\pi) = \prod_{i=0}^{a-1}\tau^{-i}A_{ f,1}(\hat{\lambda}\pi) = M_{f,1}(\hat{\lambda}\pi), 
\] 
and so for $k\geq 1$:
\begin{align}\label{line:tracelambdapi}
\Tr(\phi_{\lambda f}^{k}(\pi)) &= \Tr(M^k_{\lambda f, 1}(\pi)) = \Tr(M^k_{ f, 1}(\hat{\lambda}\pi)). 
\end{align}
But it is well-known (see Lemma~V.3.4 in \cite{Koblitz}) that
 \begin{align}\label{line:expform}
 	\det(1-\phi_f s) = \exp \left (-\sum_{i=1}^\infty \frac{\Tr(\phi_f^{i})s^i}{i}\right ),
 \end{align}
 so for each $i$, there exists a polynomial $g_i(X_1, \cdots, X_i) \in\mathbb{Q}[X_1, \cdots, X_i]$ such that $C_{ f,i}(\pi) = g_i(\Tr(\phi_{ f}(\pi)), \cdots, \Tr(\phi_{ f}^{i}(\pi)))$. However, by \eqref{line:tracelambdapi}, we see  that:
 \begin{align*}
 	C_{\lambda f,i}(\pi) &= g_i(\Tr(\phi_{\lambda f}(\pi)), \cdots, \Tr(\phi_{\lambda f}^{i}(\pi)))\\
 		&= g_i(\Tr(M_{ f,1}(\hat{\lambda}\pi)), \cdots, \Tr(M_{ f,1}^i(\hat{\lambda}\pi))) \\
 		&= g_i(\Tr(\phi_{ f}(\hat{\lambda}\pi)), \cdots, \Tr(\phi_{ f}^{i}(\hat{\lambda}\pi))) = C_{ f,i}(\hat{\lambda}\pi).
 \end{align*}
\end{proof}

This technical lemma yields our first result:
\begin{theorem}\label{main1}
Let $\lambda,\mu\in\overline{\Fp^\times}$. If $\lambda^{p-1}=\mu^{p-1}$, then $\NP(\lambda f)=\NP(\mu f)$.
\end{theorem}
\begin{proof}
By Lemma~\ref{lemma:cpi} and Proposition~\ref{prop:laminv}, and for $i\geq 0$,
\[
	|C_{\lambda f, i}(\pi)|_p = |C_{f,i}(\hat{\lambda}\pi)|_p = |C_{f,i}(\hat{\mu}\pi)|_p=  |C_{\mu f,i}(\pi)|_p,
\]
and so the Newton polygons must be the same. 
\end{proof}

When $\mu = 1$, Proposition~\ref{main1} simplifies as it did in Corollary~\ref{corr:mu_1}:
\begin{corollary}
For any $\lambda\in\Fp^\times$, the Newton polygon $\NP(\lambda f)$ is invariant. That is, $\NP(\lambda f)=\NP(f)$.
\end{corollary}

Theorem~\ref{main1} actually proves more than that the Newton polygon is invariant of $\lambda$.  It shows that the the $p$-adic valuation of the coefficients of the $L$-function do not change at all.  It is possible, and quite common when $\lambda\not\in\Fp^\times$, that the Newton polygon may not change but the $p$-adic valuations of the $L$-function coefficients shuffle around as we compute $\NP(\lambda f)$ for varying $\lambda$.

\section{A Counterexample}\label{section:counterexample}

In this section, fix $p=5$, $a=2$ and $f(x)=x^8+x^6+x^2$. We will explicitly compute the Newton polygon corresponding to $\lambda f(x)$ when $\lambda = 1$ and $\lambda = \xi+2$, where $\xi$ is again a fixed generator of $\Fq/\Fp$ so that $\Fq = \Fp(\xi)$.

\begin{remark}
While we fix $\lambda$ as $\xi+2$, any root of $\lambda^{p-1}+1$ in $\Fq$ will also yield the same result. These roots are: $\xi+2$, $2\xi + 4$,  $3\xi + 1$ and $4\xi + 3$.
\end{remark}

We begin by making some simplifications. Since $f(x)$ has no constant term, define:
\[
	L_f(\pi,s) = \frac{L_f^*(\pi,s)}{1-s},
\]
which has the same Newton polygon as $L_f^*(\pi,s)$, but without the slope zero segment. 
Observe that the first row of $A_f$ is $\begin{bmatrix} 1 & 0 & 0 & \cdots \end{bmatrix}$, and so the first row of $M_{R}(\phi_f)$ is the same. Computing $\det(1-\phi_f s)$ by expanding along this first row yields $(1-s)\det(1- \phi_f\vert_{B_0})$, where $B_0$ is the subspace of $B$ consisting of all $g(x)\in B$ with $g(0)=0$. 
Comparing this with \eqref{line:C_func_key}, we see that the first $(d-1)$-slopes of the $p$-adic Newton polygon of $\det(1-\phi_f\vert_{B_0} s)$ is the Newton polygon of $L_f(\pi,s)$. For the rest of this section, we will abuse notation and write $\NP(f)$ for the $p$-adic Newton polygon of $L_f(\pi,s)$.

Define:
\[
	\det(1-\phi_f\vert_{B_0}s) = \sum_{i=0}^{\infty}C_{i}(\pi)s^i,
\]
 and let $\alpha_1\leq\alpha_2\leq\cdots\leq\alpha_{d-1}$ be the $p$-adic slopes of the Newton polygon of $L_f(\pi,s)$. By symmetry, we must have $\alpha_4 = 1$. 
To compute the remaining slopes, we will compute the $p$-adic order of the first three coefficients of $\det(1-\phi_f\vert_{B_0} s)$ via the following: 
\begin{lemma}\label{lemma:traceform}
For an operator $\phi$, the first three coefficients of $\det(1- \phi s)$ are:
\begin{align*}
	c_1 &= -\Tr(\phi) \\
	c_2 &= \frac{1}{2}((\Tr(\phi))^2 -\Tr(\phi^2)) \\
	c_3 &= -\frac{1}{6}(\Tr(\phi))^3 + \frac{1}{2}\Tr(\phi)\Tr(\phi^2) -\frac{1}{3}\Tr(\phi^3).
\end{align*}
\end{lemma}
\begin{proof}
These equations come directly by computing the coefficients of $s$ in the formula \eqref{line:expform}. 
\end{proof}

Because the Newton polygon begins at $(0,0)$ and ends at $(d-1, d-1)$, it suffices to compute the $p$-adic valuations of the coefficients $C_i \bmod\ p^i$. By symmetry and the previous lemma, this amounts to computing  $\Tr(\phi_f\vert_{B_0}^i)\bmod\ p^3$ for $i=1,2,3$. Furthermore, since we are only computing the traces of $\phi_f\vert_{B_0}$, we can, just as we did in Lemma~\ref{lemma:cpi}, compute them with respect to the basis $R_1$. Hence we will take $A = (F_{pi-j})_{i,j\geq 1}$, where, for the sake of notational convenience,  $F_i = F_{f,i}$, and $M = M_{R_1}(\phi_f\vert_{B_0}) = A\tau A$ since  $\tau^{-1}=\tau$ when $a=2$.

Before we proceed, we will need a technical proposition in order to compute the traces with Julia.
In what follows $\ordp_\pi g(\pi)$ will denote the $\pi$-adic order of $g(\pi)$.
\begin{lemma}\label{lemma:Ffi}
For $i\geq 0$, 
\begin{align*}
	\ordp_\pi F_{i}(\pi) \geq \begin{cases}
			\ceil{i/8}  & i\textrm{ even}, i\neq 4 \\
			2 	& i = 4\\
			\infty	&\textrm{ otherwise}
		\end{cases}.
\end{align*}
\end{lemma}
\begin{proof}
Clearly,  when $i$ is odd,  $F_{i}$ is zero.  Note that for $i\equiv 2\textrm{ or }6\bmod 8$, the bound is clear. For $i\equiv 4\bmod 8$, and $i\neq 4$, we see that we can write $i=8(\floor{\frac{i}{8}} -1) + 2\cdot 6$ and so the bound also holds in this case. 
\end{proof}

We have the following $\pi$-adic bound for the diagonal entries of the matrix $M^k$. 
\begin{prop}\label{prop:matrixbound}
For $k, i\geq 1$:
\begin{align*}
	\ordp_\pi (M^k)_{i,i} \geq\begin{cases}
		2 & \textrm{ if } k=i=1\\
		\ceil{\frac{pi-1}{8}} + 2k-1 & \textrm{ otherwise.}
	\end{cases}
\end{align*}
\end{prop}
\begin{proof}
When $k=i=1$, 
\[
	(M)_{1,1} = \sum_{k=1}^p F_{p-k}\tau F_{pk-1}.
\]
If $\ordp_\pi  F_{p-k}\tau F_{pk-1} = 1$, we must have $k=p$ since $\ordp_\pi F_{pk-1}\geq 1$. But in this case, $\ordp_\pi F_{p^2-1} = 3>1$ and so $\ordp_\pi (M)_{1,1}\geq 2$. 

In the general case, observe that the $(i,j)$th entry of the matrix $M^k$ can be given by 
\[
	(M^k)_{i,j} = \sum_{\substack{w_1, \cdots, w_{k-1}\geq 1\\w_0=i, w_k=j}} \prod_{\ell = 1}^k (M)_{w_{\ell-1}, w_\ell}, 
\]
so that:
\[
	(M^k)_{i,i} =  \sum_{\substack{w_1, \cdots, w_{k-1}\geq 1\\ w_0 = i, w_k=i}} \prod_{\ell = 1}^k \left (\sum_{v = 1}^\infty F_{pw_{\ell-1}-v} \tau F_{pv-w_\ell}\right ).
\]
Take some $w_0, \cdots, w_k$ and $v_1, \cdots, v_k$. In the expansion of $(M^k)_{i,i}$, each monomial comes from a product of the form
\[
	T = \prod_{j=1}^k F_{pw_{j-1}-v_j} \tau F_{pv_j-w_j},
\]
and so 
\begin{align*}
	\ordp_\pi T &\geq \sum_{j=1}^k\left (\ceil{\frac{pw_{j-1}-v_j}{8}} + \ceil{\frac{pv_j-w_j}{8}}\right )\\
		&\geq \ceil{\frac{pi-1}{8}} + \ceil{\frac{p-1}{8}} + 2\sum_{j=2}^{k-1}\ceil{\frac{p-1}{8}} + (\ceil{\frac{p-1}{8}}+ \ceil{\frac{pv_k-i}{8}})\\
		&\geq \ceil{\frac{pi-1}{8}}  + 2k-1.
\end{align*}
\end{proof}

\subsection{The Case $\lambda = 1$} 

\begin{lemma}\label{1_traces}
For $\lambda = 1$, 
\begin{alignat*}{2}
\Tr(M) &\equiv 20 + 115\pi+ 96\pi^{2} + 66\pi^{3}\bmod p^3\\
\Tr(M^2) &\equiv 95+ 105\pi + 60\pi^{3}\bmod p^3\\
\Tr(M^3) &\equiv 100 + 25\pi + 70\pi^{2} + 45\pi^{3}\bmod p^3
\end{alignat*}
\end{lemma}
\begin{proof}
In this proof, we explain our algorithm that computes these traces in Julia.

For $\pi$ to be a root of $E(\pi) = \zeta_p$, it must satisfy $\sum_{k=0}^\infty \frac{\pi^{p^k}}{p^k} = 0$ (see Lemma~2.2 in \cite{LiuWei}). Therefore, because 
\[
	\ord{\frac{\pi^{p^i}}{p^i}} = \frac{p^i}{p-1}-i,
\]
the order of each successive term in the series is increasing, and so when $p=5$, $\pi^{p-1} + p\equiv 0\bmod p^3$. Thus, for the sake of computation modulo $p^3$, we can simply treat $\pi$ as a formal variable   such that $\pi^{p-1} = -p$, and operate in the quotient polynomial ring $T = (\Zp/p^3\Zp)[\pi]/(\pi^{p-1} + p)$, equipped with the valuation 
\[
	\ordp_{}{\left (\sum_{i=0}^{p-2}\pi^i{h}_i\right )} = \min_{0\leq i\leq p-1}\left (\ord{{h}_i} + \frac{i}{p-1}\right ).
\]

Note that by Proposition~\ref{prop:matrixbound}, modulo $p^3$, it is enough to compute the trace of the $20\times 20$ principal minor for each $M^i$, $i=1,2,3$. For example, for $M^1$, Proposition~\ref{prop:matrixbound} says that any terms on the diagonal of $M$ after the $20$th term has $\pi$-adic order greater than $13.375$, which would vanish $\bmod\ p^3$. Thus, we get the following algorithm:

\begin{itemize}
	\item [{\it Step 0:}] Let $x$ be a formal variable and compute the Artin-Hasse exponential $E(x) = \sum_{k=0}^\infty u_kx^k$ modulo $x^{20p}$ with coefficients in $\Zp/p^3\Zp$. We need to compute the first $(20p-1)$-many coefficients of this power series because our matrix in a later step uses the coefficients of $E(x)$ up through $u_{20p-1}$.
	\item [{\it Step 1:}] Compute $\hat{\lambda}$ in $T$ and then calculate:
		\[
			F = E(\pi\hat{\lambda} x^8) E(\pi\hat{\lambda} x^6) E(\pi\hat{\lambda} x^2) = \sum_{i=0}^{12p - 1}F_ix^i
		\]
		 in $T[x]\bmod\ x^{20p}$. We compute the coefficients $\bmod\ x^{20p}$ for the same reason as in {\it Step~0}. 
	\item [{\it Step 2:}] Set $A = [F_{pi-j}]_{1\leq i,j\leq 20}$, where we put $F_{pi-j}=0$ if $pi-j<0$. 
	\item [{\it Step 3:}] Compute $M = A\tau A$.
	\item [{\it Step 4:}] Compute $M^i$ and then $\Tr(M^i)$. This will output the trace as an element in $T$.
\end{itemize}
For the exact implementation, see the Github repository in the introduction. 
\end{proof}

\begin{prop}\label{prop:lemmaonec23}
When $\lambda = 1$, $\ord C_1 = \frac{1}{2}$,  $\ord C_2 = 1.25$ and $\ord C_3 = 2$. 
\end{prop}
\begin{proof}
This is a simple calculation following from Lemma~\ref{lemma:traceform} and Lemma~\ref{1_traces}. The $p$-adic order of $C_1$ follows directly from Lemma~\ref{1_traces}, and $C_2$ follows from the computation: 
\begin{align*}
	C_2 & \equiv \frac{1}{2}((20+15\pi+21\pi^2+16\pi^3)^2 - (20+5\pi+10\pi^2))\\
		& \equiv 5\pi + 5\pi^{2} + 5\pi^{3}\bmod\ p^2. 
\end{align*}
Similarly computing the third coefficient yields $C_3 \equiv 25 + 50\pi + 25\pi^2\bmod p^3$. 
\end{proof}

To compute the Newton polygon from just these three coefficients, we will need to use the symmetry of the slopes in the form of the following lemma. 

\begin{lemma}\label{lemma:pointsymm}
Let $L_f(s) = \sum_{i=0}^{d-1} c_i s^i$ with $t_i=\ord c_i$. Then $t_i\geq t_{(d-1)-i}+ai-(d-1)$.
\end{lemma}
\begin{proof}
For notational convenience, write $D=d-1$. We proceed by induction. When $i=0$, $t_0 = t_{D} + a(0) - D = D-D=0$, so the inequality clearly holds in this case. 

Now suppose that the claim holds for $i$ so that $t_i \geq t_{D-i} + ai- D$. Write the $L$-function in terms of its roots,
\[
	L_f(s) = c_0\prod_{i=1}^{D}(1-r_is),
\]
where without loss of generality we order the roots $\ord r_1\leq \ord r_2\leq \cdots \leq \ord r_{D}$, noting that each $\ord r_i$ is one of the slopes of the Newton polygon. The coefficients of $L_f(s)$ are related to the roots by
\[
	c_i = (-1)^{i}c_{0}\sum_{1\leq \epsilon_1<\cdots<\epsilon_{i}\leq D}\prod_{j=1}^{i}r_{\epsilon_j}.
\]
So $c_{i+1}$ is a sum of products of $i+1$ roots while $c_i$ is a sum of products of $i$ roots. 
Therefore, since we ordered the roots of $L$, we see that $t_{i+1}\geq t_i+\ord r_{i+1}$ and similarly, $t_{D-(i+1)} \leq t_{D-i} - \ord r_{D-(i+1)}$. But by induction and the symmetry of the slopes:
\begin{align*}
	t_{D-(i+1)} - t_{i+1} &\leq (t_{D-i} - \ord r_{D-(i+1)}) - (t_i+\ord r_{i+1})\\
		&\leq t_{D-i}-\ord r_{D-(i+1)} - (t_{D-i}+(ai-D))-\ord r_{i+1}\\
		&= -\ord r_{D-(i+1)} -(ai-D)-(a-\ord r_{(D-(i+1)})\\
		&= -(a(i+1)-D).
\end{align*}
\end{proof}

\begin{prop}\label{prop:NP1}
The $p$-adic Newton polygon $\NP(f)$ has slopes (with multiplicities): 
\[
	\{(0.5, 1.0), (0.75, 2.0), (1.0, 1.0), (1.25, 2.0), (1.5, 1.0)\}.
\]
\end{prop}
\begin{proof}
We know the $p$-adic valuation of the first four coefficients of $\det(1-\phi\vert_{B_0} s)$ by Proposition~\ref{prop:lemmaonec23}, noting that $\ord C_0 = 0$. By Proposition~\ref{prop:matrixbound}, it is easy to see that $\ordp_\pi\Tr M^k\geq 2k$, and so the Newton polygon must lie above the lower bound $y=\frac{1}{2}x$. However, combined  with our computation that $\ord{C_1} = \frac{1}{2}$, this implies that the first slope must be $\frac{1}{2}$. Consequently the last slope must be $1.5$ and by symmetry again, the fourth slope must be $1$. It remains to find the second and third slopes (and hence the fifth and sixth slopes).

By Lemma~\ref{lemma:pointsymm}, we have $\ord{C_4}\geq 3$, $\ord{C_5}\geq 4.25$ and $\ord{C_6}\geq 5.5$, which coincide with the vertices on Figure~\ref{figure:one}. Therefore, when computing the lower convex hull, the points $\ord{C_2}=1.25$ and $\ord{C_3}=2$ will necessarily be breaking points and the second and third slopes are as claimed.
\end{proof} 

Thus, we have the Newton polygon in Figure~\ref{figure:one}. Note that the lower bound in red is the Hodge polygon (see \cite{wanexpo}, Definition 8.9), and the vertices are the $p$-adic orders of the coefficients of the $L$-function, computed precisely via Lauder and Wan's algorithm.
\begin{figure}
\centering
\includegraphics[scale=.25]{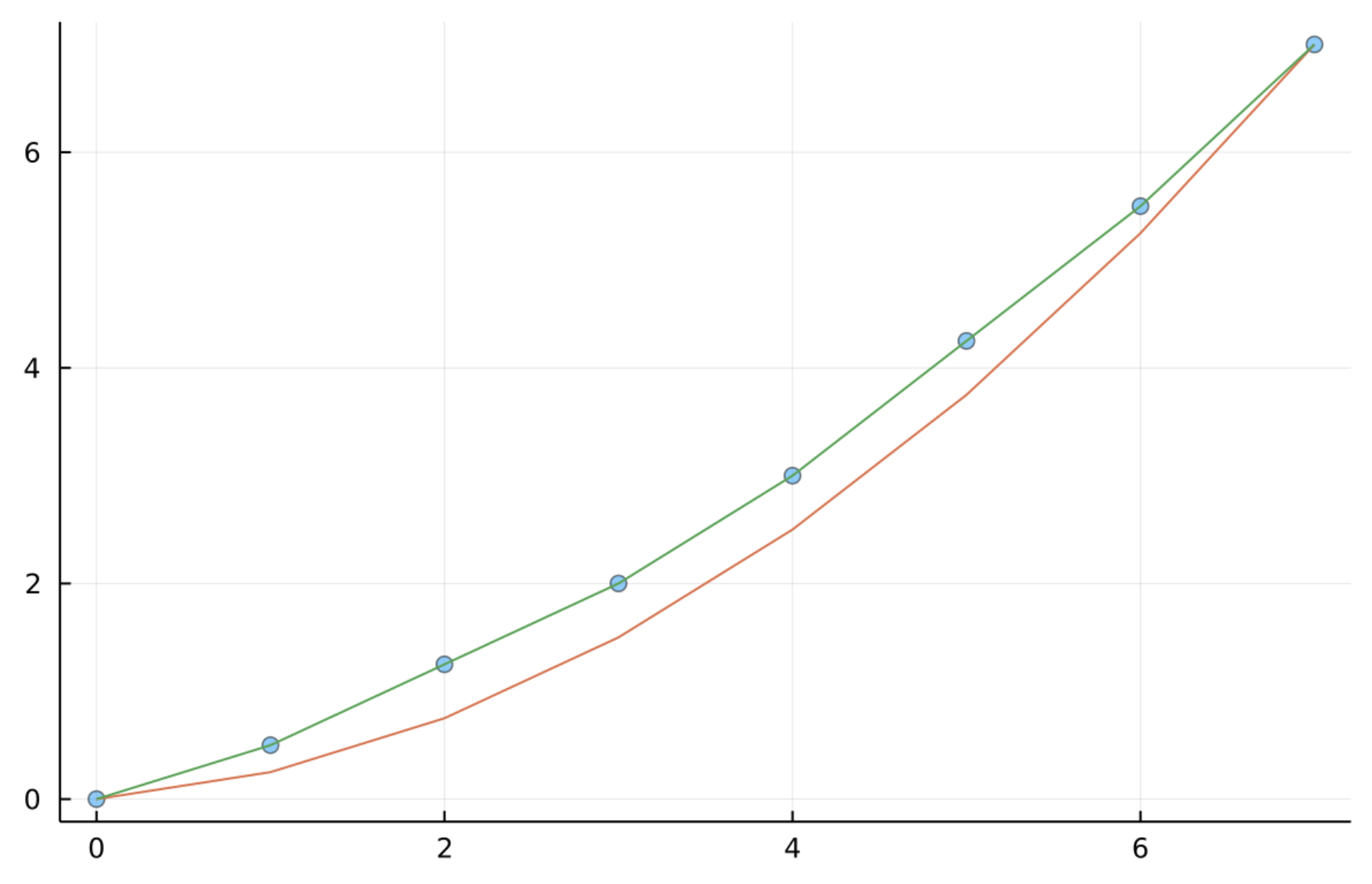}
\caption{$\NP(f)$ for $f(x)=x^8+x^6+x^2$}
\label{figure:one}
\end{figure}

\subsection{The Case $\lambda = \xi+2$}\label{section:xi2}

When $\lambda = \xi+2$, a key difference that occurs is that the action of $\tau$ on the matrix $\phi$ is nontrivial. 
Observe 
\[
	(\phi)_{i,j} = \sum_{k=1}^\infty F_{pi-k}(\hat{\lambda} \pi)\tau F_{pk-j}(\hat{\lambda}\pi)=\sum_{k=1}^\infty F_{pi-k}(\hat{\lambda} \pi)F_{pk-j}(\hat{\lambda}^p\pi).
\]

\begin{lemma}\label{lemma:p1lift}
If $\lambda = p-1\in\Fp$, then $\widehat{\lambda} = -1\in\Zptimes$.
\end{lemma}
\begin{proof}
Let $t=p-1\in\Zptimes$. It is well-known that the \teich lifting of $p-1\in\Fp$ can be computed by
\[
	\widehat{p-1} = \lim_{i\to\infty}t^{p^i},
\]
where $t^{p^i}$ computes the \teich lifting to $p$-adic accuracy $p^{i+1}$. Write $\hat{\lambda} = \sum_{i=0}^\infty a_ip^i$. We claim $a_i=p-1$ for all $i$.

Proceed by induction on $i$. For $i = 0$, $a_0 = t^{1} = p-1$. So assume that $a_{k} = p-1$ for $k< i$. 
Consider that:
\begin{align*}
	(p-1)^{p^i} &= \sum_{k=0}^{p^i}\binom{p^i}{k}p^{p^i-k}(-1)^k
	\equiv \sum_{k=0}^i\binom{p^i}{p^i-k}p^k(-1)^{p^i-k}
	\\&\equiv (-1)^{p^i} + \sum_{k=1}^i\frac{\binom{p^i}{p^i-k}}{p^i}p^{i+k}(-1)^{p^i-k}\bmod\ p^{i+1},
\end{align*}
where the last congruence follows from the observation that for $k\geq 1$, $p^k\mid\binom{p^k}{p^k-i}$. The only remaining term $\bmod\ p^{k+1}$ is $(-1)^{p^k}$, which, because $p$ is an odd prime, yields $-1$. 
Therefore,  $t^{p^i} = (p-1)^{p^i}\equiv p^{i+1}-1\bmod p^{i+1}$. However, using the induction assumption and solving for $a_i\bmod\ p^{i+1}$ yields:
\begin{align*}
	a_i = \frac{(p^{i+1}-1) - \sum_{i=0}^{i-1}(p-1)p^i}{p^i} = \frac{(p^{i+1}-1)-(p^i-1)}{p^i} = p-1.
\end{align*}
\end{proof}

\begin{lemma}\label{lemma:sumzero}
Let $\lambda\in\Fqtimes$ be a root of $x^{p-1}+1$ and let $w\in\Zqtimes$ be a \teich lifting of $\lambda$. For $b,c$ positive integers with $b+c$ odd:
\[
	w^{b+pc}+w^{c+pb} = 0.
\]
\end{lemma}
\begin{proof}
It is clear that $w^{b+pc}+w^{c+pb} = w^{b+c} (w^{(p-1)c} + w^{(p-1)b})$.  But because $b+c$ is odd, we can, without loss of generality, take $c$ to be odd and $b$ to be even. Then, by Lemma~\ref{lemma:p1lift}:
\begin{align*}
	w^{(p-1)c} + w^{(p-1)b} &= (w^{p-1})^{2k+1}+(w^{p-1})^{2k'} = (-1)^{2k+1}+(-1)^{2k'} \\
				&= -1  + 1=0,
\end{align*}
since $w^{p-1}=\hat{\lambda}^{p-1} = \widehat{\lambda^{p-1}} = \widehat{p-1}= -1$.
\end{proof}

For a product of power series $G = g_1\cdots g_k$ we refer to the $(r_1, \cdots, r_k)$-monomial, $[g_1\cdots g_k]_{(r_1,\cdots,r_k)}$, as the monomial obtained from the product $G$ by taking the $r_1$-power term from $g_1$, the $r_2$-power term from $g_2$, etc. 
\begin{lemma}\label{lemma:sumonlyeven}
Let $G(x), H(x)\in \Qq[[x]]$. Then the sum $G(\hat{\lambda} x)H(\hat{\lambda}^p x) + G(\hat{\lambda}^p x)H(\hat{\lambda} x)$ has only even powers of $x$.
\end{lemma}
\begin{proof}
Write $G(x) = \sum_{k=0}^\infty G_k x^k$ and $H(x)=\sum_{k=0}^\infty H_k x^k$. For $r_1,r_2\geq 1$ with $r_1+r_2$ odd, we compute that
\begin{align*}
	[G(\hat{\lambda} x)H(\hat{\lambda}^p x)]_{(r_1,r_2)} &+ [G(\hat{\lambda}^p x)H(\hat{\lambda} x)]_{(r_1,r_2)}  \\
		 &=G_{r_1}\hat{\lambda}^{r_1}x^{r_1}H_{r_2}\hat{\lambda}^{pr_2}x^{r_2} + G_{r_1}\hat{\lambda}^{pr_1}x^{r_1}H_{r_2}\hat{\lambda}^{r_2}x^{r_2}\\
		 & =(\hat{\lambda}^{r_1+pr_2}+\hat{\lambda}^{r_2+pr_1}) G_{r_1}H_{r_2}x^{r_1+r_2}\\
		 &=0,
\end{align*}
where the last equality follows from Lemma~\ref{lemma:sumzero} and the fact that $\hat{\lambda}^{p-1}+1=0$.
\end{proof}

\begin{prop}\label{prop:even_trace}
	For $k\geq 1$, $\Tr(M^k)$ has only even powers of $\pi$.
\end{prop}
\begin{proof}
Fix $k\geq 1$. 
Just as in the proof of Proposition~\ref{prop:matrixbound}, 
\[
	(M^k)_{i,i} =  \sum_{\substack{w_1, \cdots, w_{k-1}\geq 1\\w_0=w_k=i}} \prod_{\ell = 1}^k \left (\sum_{v = 1}^\infty F_{pw_{\ell-1}-v}(\lambda\pi) F_{pv-w_\ell}(\lambda^p\pi)\right ),
\]
and so 
\[
	\Tr M^k = \sum_{i=1}^\infty\sum_{w_1, \cdots, w_{k-1}\geq 1} \prod_{\ell = 1}^k \left (\sum_{v = 1}^\infty F_{pw_{\ell-1}-v}(\lambda\pi) F_{pv-w_\ell}(\lambda^p\pi)\right ).
\]
Expanding $\Tr M^k$, each monomial comes from a product of the form:
\[
	T = \prod_{j=1}^k F_{pw_{j-1}-v_j}(\lambda\pi) F_{pv_j-w_j}(\lambda^p\pi),
\]
some $v_1, \cdots, v_k\geq 1$ and $w_0, \cdots, w_k\geq 1$, with $w_0=w_k=i$ and $i\geq 1$. 

Say we have a $(r_1, \cdots, r_{2k})$-monomial coming from a product 
\[
	\prod_{j=1}^k F_{pw_{j-1}-v_j} F_{pv_j-w_j},
\]
some fixed $w_j$, $v_j$ and $i$. Define a new set of indices $w_j' = v_j$ and $v_j'=w_{j-1}$. These indices yield a product:
\[
	T' = \prod_{j=1}^k F_{pv_{j-1}-w_{j-1}}(\lambda\pi) F_{pw_{j-1}-v_j}(\lambda^p\pi).
\]
Define $G(x) =  \prod_{j=1}^k F_{pv_{j-1}-w_{j-1}}(x)$ and $H(x) = \prod_{j=1}^k F_{pw_{j-1}-v_j}(x)$.
Then, by Lemma~\ref{lemma:sumonlyeven}, adding the $(r_1, \cdots, r_{2k})$-monomial from $T$ to the $(r_{2k}, r_1, \cdots, r_{2k-1})$-monomial of $T'$ yields zero, and the proposition follows.
\end{proof}

To see this explicitly, if we compute the traces of $M$ and its powers when $\lambda = \xi+2$, all the odd powers of $\pi$ collapse.
\begin{lemma}\label{lemma:xi2_traces}
When $\lambda = \xi + 2$,
\begin{align*}
\Tr(M) &\equiv 15 + 68\pi^{2}\bmod p^3\\
\Tr(M^2) &\equiv 30+ 90\pi^{2} \bmod p^3\\
\Tr(M^3) &\equiv 100 + 90\pi^{2} \bmod p^3
\end{align*}
\end{lemma}
\begin{proof}
See proof of Lemma~\ref{1_traces}.
\end{proof}

\begin{lemma}\label{lemma:xi2_c2c3}
We have  $\ord C_1 =\frac{1}{2}$, $\ord C_2\geq 2$ and $\ord C_3\geq  3$. 
\end{lemma}
\begin{proof}
Just like Lemma~\ref{prop:lemmaonec23}, this is a calculation following from Lemma~\ref{lemma:traceform} and Lemma~\ref{lemma:xi2_c2c3}. In this case, we get that $C_1 \equiv 2\pi^2\bmod p$, $C_2\equiv 0\bmod p^2$ and $C_3\equiv 0\bmod\ p^3$. 
\end{proof}

\begin{figure}[H]
\centering
\includegraphics[scale=.25]{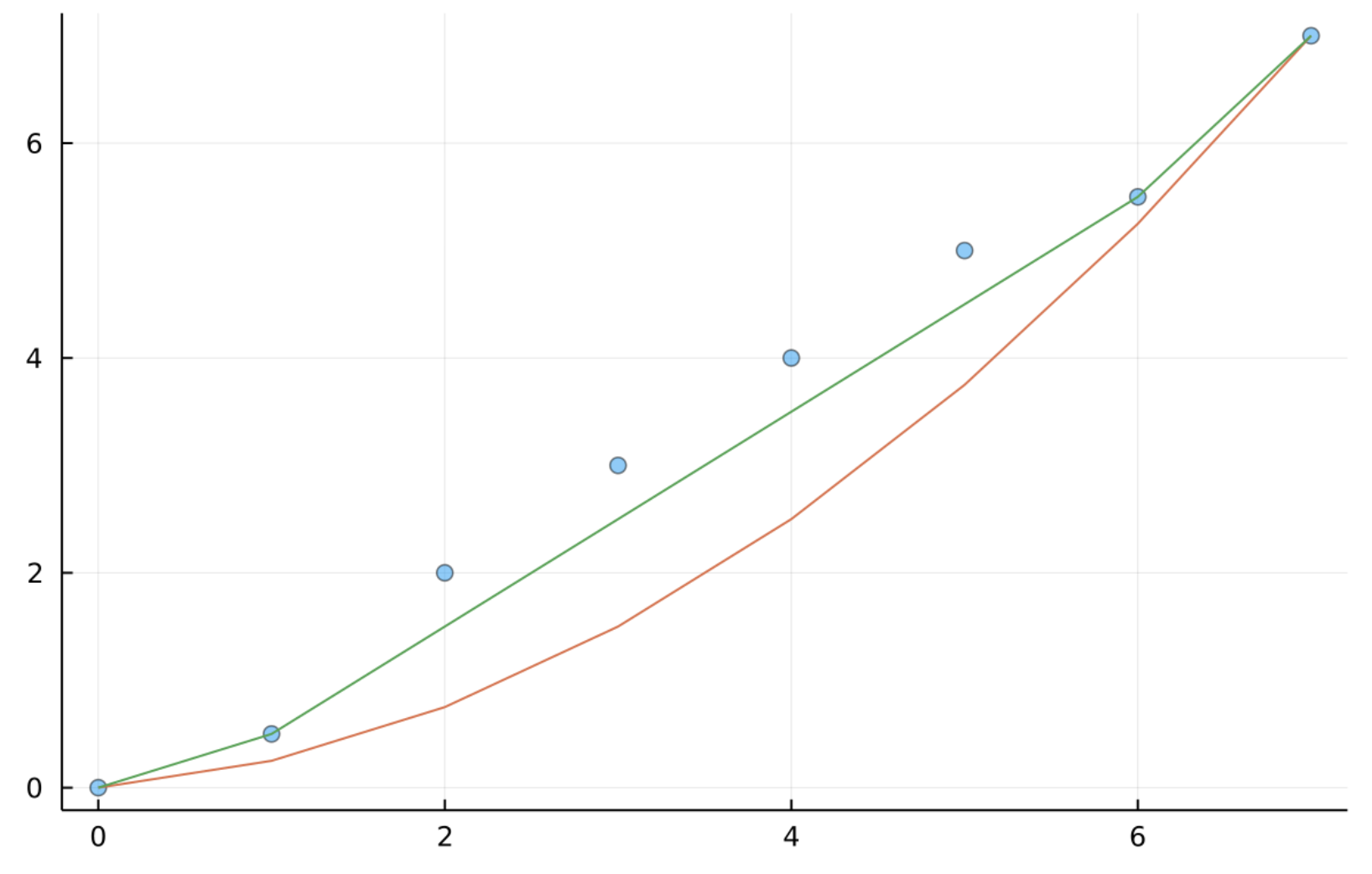}
\caption{$\NP(\lambda f)$ for $f = x^8+x^6+x^2$, $\lambda = \xi+2$}
\end{figure}

\begin{prop}
When $\lambda = \xi +2$, the Newton polygon $\NP(\lambda f)$ has $p$-adic slopes (with multiplicities):
\[
	\{(0.5, 1.0), (1.0, 5.0), (1.5, 1.0)\}.
\]
\end{prop}
\begin{proof}
Just as we did in Proposition~\ref{prop:NP1}, by Lemma~\ref{lemma:pointsymm}, 
$\ord{C_4}\geq 4$, $\ord{C_5}\geq 5$ and $\ord{C_6}\geq 5.5$. But by symmetry, since the first slope is again $\frac{1}{2}$, the last slope is $2-\frac{1}{2} = 1.5$. However, because $\ord C_i\geq i$ for $i=2,3,4,5$, the middle slopes must be $1$.
\end{proof}

\end{document}